\documentclass[12pt,english]{amsart}
\usepackage{amsmath,amsthm,amssymb,amsfonts,amscd, array,latexsym, pb-diagram}
\usepackage[noadjust]{cite}
\usepackage[T2A]{fontenc}
\usepackage[cp1251]{inputenc}
\usepackage[english]{babel}
\usepackage{diagmac2} 
\usepackage{tikz}

\usepackage{color} 

\newtheorem{theorem}{Theorem}[section]
\newtheorem{corollary}[theorem]{Corollary}
\newtheorem{lemma}[theorem]{Lemma}
\newtheorem{proposition}[theorem]{Proposition}
\newtheorem{conjecture}[theorem]{Conjecture}

\theoremstyle{definition}
\newtheorem{definition}[theorem]{Definition}
\newtheorem{example}[theorem]{Example}

\theoremstyle{remark}
\newtheorem{remark}[theorem]{Remark}
\numberwithin{equation}{section}

\DeclareMathOperator{\Id}{Id}
\DeclareMathOperator{\In}{In}

\newcommand{\RR}{\mathbb{R}}

\def\<#1>{\langle #1 \rangle}

\newbox\onebox
\newcommand{\coherent}[1]{\mathbin{\setbox\onebox=\hbox{$=$}\lower0.7\ht%
\onebox\hbox{$\stackrel{#1}{=}$}}}

\newcommand{\acr}{\newline\indent}

\begin{document}

\title{Ultrametric-preserving functions as monoid endomorphisms}

\author{Oleksiy Dovgoshey}
\address{\textbf{Oleksiy Dovgoshey}\acr
Department of Theory of Functions \acr
Institute of Applied Mathematics and Mechanics of NASU \acr
Slovyansk, Ukraine, \acr
Department of Mathematics and Statistics \acr
University of Turku,  Turku, Finland}

\email{oleksiy.dovgoshey@gmail.com, oleksiy.dovgoshey@utu.fi}

\begin{abstract}
 Let $\mathbb{R}^{+}=[0, \infty)$ and let  $\mathbf{End}_{\RR^+}$ be the set of all endomorphisms of the monoid $(\RR^+, \vee)$. The set $\mathbf{End}_{\RR^+}$ is a monoid with respect to the operation of the function composition $g \circ f$. It is shown that $g : \RR^+ \to \RR^+$ is pseudoultrametric-preserving  iff $g \in \mathbf{End}_{\RR^+}$. In particular, a function $f : \RR^+ \to \RR^+$ is ultrametrics-preserving iff it is an endomorphism of $(\RR^+,\vee)$ with kelnel consisting only the zero point.  We prove that a given $\mathbf{A} \subseteq \mathbf{End}_{\RR^+}$  is a submonoid of $(\mathbf{End}, \circ)$ iff there is a class $\mathbf{X}$ of pseudoultrametric spaces such that $\mathbf{A}$ coincides with the set of all functions which preserve the spaces from $\mathbf{X}$. An explicit construction of such $\mathbf{X}$ is given.
\end{abstract}



\maketitle

\section{Introduction. Ultrametrics and pseudoultrametrics}

Let $\RR^{+}$  denote the set of all nonnegative real numbers.

\begin{definition}\label{def1.1}
Let $X$ be a nonempty set. A \emph{metric} on a set $X$ is a function $d : X \times X \to \RR^{+}$ such that for all $x, y, z \in X$:
\begin{itemize}
\item[(i)] $(d(x, y) = 0) \Longleftrightarrow (x = y)$, the \emph{positivity property};
\item[(ii)] $d(x, y) = d(y, x)$, the \emph{symmetry property};
\item[(iii)] $d(x, y) \leqslant d(x, z) + d(z, y)$, the \emph{triangle inequality}.
\end{itemize}
\end{definition}

A metric $d : X \times X \to \RR^{+}$ is an \emph{ultrametric} on $X$ if the \emph{strong triangle inequality}
\begin{itemize}
\item[(iv)] $d(x, y) \leqslant \max\{d(x, z), d(z, y)\}$
\end{itemize}
holds for all $x,y,z \in X$.

\begin{example}\label{ex1.2}
Following  \cite{DLPS2008TaiA} we define a mapping $ d^+ : \mathbb{R}^+ \times \mathbb{R}^+ \to \mathbb{R}^+$ as
\begin{equation}\label{ex1.2_eq1}
            d^+(p,q) := \left\{
  \begin{array}{ll}
    0,  \quad & \hbox{if} \quad p=q, \\
    \max\{p,q\}, \quad & \hbox{otherwise.}
  \end{array}
\right.
\end{equation}
Then $d^+$ is an ultrametric on $\RR^+$.
\end{example}

Different properties of  ultrametric spaces have been studied in \cite{BS2017, DLPS2008TaiA, DD2010, DM2008, DM2009, DP2013SM, GH1961S, GroPAMS1956, Ibragimov2012, KS2012, Lemin1984FAA, Lemin1984RMS39:5, Lemin1984RMS39:1, Lemin1985SMD32:3, Lemin1988, Lemin2003, PTAbAppAn2014, Qiu2009pNUAA, Qiu2014pNUAA, Vau1999TP, Ves1994UMJ, Ish2021ANUAA, Ish2023TA, Ish2023ANUAA, Ish2024TA, SP2024FPTASE, SKPMS2020, DovBBMSSS2020}.

The useful generalization of the concept of ultrametric is the concept of pseudoultrametric.

\begin{definition}\label{def2.1}
Let $X$ be a nonempty set and let $d : X \times X \to \RR^{+}$ be a symmetric function such that $d(x, x) = 0$ holds for every $x \in X$. The function $d$ is a \emph{pseudoultrametric} on $X$ if it satisfies the strong triangle inequality.
\end{definition}

If $d$ is a pseudoultrametric on $X$, then we will say that $(X, d)$ is a \emph{pseudoultrametric space}.

Every ultrametric space is a pseudoultrametric space but not conversely. In contrast to ultrametric spaces, pseudoultrametric spaces can contain some distinct points with zero distance between them.

\begin{example}\label{ex1.4}
Let $(X,d)$ be an ultrametric space  and let  $k \in \mathbb{R}^+$. Then the mapping $d_k : X \times X  \to \RR^+$,
$$
d_k (x,y) = \left\{
  \begin{array}{ll}
    d(x,y), \quad & \hbox{if} \quad d(x,y) > k \\
    0, \quad & \hbox{if} \quad   d(x,y) \leqslant k
  \end{array}
\right.
$$
is a pseudoultrametric on $X$.
\end{example}

The following concept was introduced in \cite{PTAbAppAn2014}.

\begin{definition}\label{def2.8}
A function $f : \RR^+ \to \RR^+$ is \emph{ultrametric-preserving} if $f \circ d$ is an ultrametric for every ultrametric space $(X, d)$.
\end{definition}

\begin{remark}\label{rem1.2}
Here and what follows we write $f \circ d$ for the mapping
$$
X \times X \overset{d}{\to} \RR^+ \overset{f}{\to} \RR^+ .
$$
\end{remark}

The next definition is an extension of Definition~\ref{def2.8} to pseudoultrametric spaces.

\begin{definition}\label{def2.3}
A function $f: \RR^{+} \to \RR^{+}$ is \emph{pseudoultrametric-pre\-ser\-ving} if $f \circ d$ is a pseudoultrametric for every pseudoultrametric space $(X,d)$.
\end{definition}

We denote by $\mathbf{U}$ the class of all ultrametric spaces and by $\mathbf{PU}$ the class of all pseudoultrametric spaces. We also will use the designations $\mathbf{P_U}$ and $\mathbf{P}_{\mathbf{PU}}$ for the sets of all ultrametric-preserving functions and, respectively, for the set of all  pseudometric-preserving functions. The next definition is a pseudoultrametric modification of the concept introduced by Jacek Jachymski and Filip Turobo\'s in \cite{JTRRACEFNSAMR2020}.

\begin{definition}\label{def1.8}
Let $\mathbf{X}$ be a subclass of the class $\mathbf{PU}$. Let us denote by $\mathbf{P_X}$ the set of all functions $f: \RR^+ \to \RR^+$ such that $f \in \mathbf{P_X}$ holds if and only if we have $(Y, f \circ \rho) \in \mathbf{X}$ whenever $(Y, \rho) \in \mathbf{X}$.
\end{definition}

The main goal of the present paper is to give new characterizations of the sets $\mathbf{P_U}$ and $\mathbf{P}_{\mathbf{PU}}$  and find conditions under which the equation
\begin{equation}\label{s1_eq1}
\mathbf{P_X} = \mathbf{A}
\end{equation}
has a solution  for given $\mathbf{A} \subseteq \mathbf{P}_{\mathbf{PU}}$.

The paper is organized as follows. The next section contains some definitions and facts from semigroup theory and some results related to ultrametric-preserving functions and pseudoultrametric-preserving ones.

The main results of the paper are proved in Section~3. In Theorem~\ref{mainth} we prove that the pseudoultrametric-preserving functions coincide with endomorphisms of the monoid $(\RR^+, \vee)$.
Theorem~\ref{th3.4} contains a characterization of endomorphisms of $(\RR^+, \vee)$ belonging to $\mathbf{P_U}$. Theorem~\ref{th3.5}  shows that equation~\eqref{s1_eq1} has a solution $\mathbf{X} \subseteq \mathbf{PU}$ iff $\mathbf{A}$ is a submonoid of $\mathbf{P_{PU}}$.  An explicit solution to equation\eqref{s1_eq1} is described in Propositions~\ref{prop3.6} and \ref{prop3.11}.

\section{Preliminaries on ultrametric-preserving  functions  and semigroups}

Recall that $f : \RR^{+} \to \RR^+$ is \emph{increasing} if
$$
(a \geqslant b) \Longrightarrow (f (a) \geqslant f (b))
$$
holds for all $a, b \in \RR^+$. A function $f : \RR^{+} \to \RR^+$ is said to be \emph{amenable} if $f^{-1}(0)=\{0\}$.

The present paper is mainly  motivated by the following result  by P.~Ponsgriiam and I. Termwuttipong \cite{PTAbAppAn2014}.

\begin{theorem}\label{th2.9}
A function $f : \RR^+ \to \RR^+$ is ultrametric-preserving if and only if $f$ is amenable and increasing.
\end{theorem}

The next extension of Theorem~\ref{th2.9} was given in \cite{Dov2020MS}.

\begin{proposition}\label{pr2.4}
The following conditions are equivalent for every function $f : \RR^+ \to \RR^+$.
\begin{itemize}
\item[(i)] $f$ is increasing and $f (0) = 0$ holds.
\item[(ii)] $f$ is pseudoultrametric-preserving.
\end{itemize}
\end{proposition}

Let us recall some basic concepts of semigroup theory, see, for example, ``Fundamentals of Semigroup Theory'' by John M. Howie \cite{Howie2003}.

A \emph{semigroup} is a pair $(S, \ast)$ consisting of a nonempty set $S$ and an associative operation $\ast : S \times S \to S$ which is called the \emph{multiplication} on $S$. A semigroup $S = (S, \ast)$ is a \emph{monoid} if there is $1_S \in S$ such that
$$1_S \ast s = s \ast 1_S = s$$
for every $s \in S$.

Let us give now an example of a monoid, which will be basic for the next section of the paper.

\begin{example}\label{ex2.11}
Let us define the binary operation $\vee$ on the set $\RR^+$ as
\begin{equation*}\label{ex2.3_eq1}
x \vee y  = \max \{x,y\}
\end{equation*}
for all $x,y \in \RR^+$. Then $(\RR^+, \vee)$ is a monoid with the identity element
\begin{equation}\label{ex2.3_eq2}
1_{\RR^+}=0.
\end{equation}
To see that $(\RR^+,\vee)$ really is a monoid and \eqref{ex2.3_eq2} holds, it suffices to note that the equalities
$$
\max \{ \{x,y\}, z\} = \max \{x,y,z\} = \max \{x, \{y,z\}\}
$$
and
$$
\max \{x,0\} = x
$$
are satisfied for all $x,y,z \in \RR^+$.
\end{example}

\begin{definition}\label{def2.5*}
Let $S=(S, \cdot, 1_S)$ and $T= (T, \ast, 1_T)$ be monoids with identity elements $1_S$ and $1_T$ respectively. A mapping $\Phi: S \to T$ is called a \emph{homomorphism} if, for all $x,y \in S$, we have
$$
\Phi (x \cdot y) = \Phi (x) \ast \Phi (y)
$$
and the equality
$$
\Phi (1_S) = 1_T
$$
holds.
\end{definition}

A homomorphism from $S$ to $S$ is called an \emph{endomorphism} of $S$.

\begin{proposition}\label{pr2.12}
Let $S= (S, \ast, 1_S)$ be a monoid and let
$\mathbf{End}_S$ be the set of all endomorphisms $S \to S$. Then $\mathbf{End}_S$ is a monoid with respect to the operation of the function composition $f \circ g$. The function $\mathrm{Id}_S : S \to S$, $\mathrm{Id}_S (s) = s$ for each $s \in S,$ is the identity of the monoid $(\mathbf{End}_S, \circ)$,
$$
1_{\mathbf{End}_S} = \mathrm{Id}_S.
$$
\end{proposition}

\begin{proof}
The composition of mappings is associative and
$$
f= \Id_S \circ f = f \circ \Id_S
$$
for every $f : S \to S$. Thus $(\mathbf{End}_S, \circ)$ is a  monoid if we have
\begin{equation}\label{pr2.5_preq1}
f \circ g  \in \mathbf{End}_S
\end{equation}
whenever
$$
f, g \in \mathbf{End}_S.
$$
Suppose that given $f : S \to S$ and $g: S \to S$  belong  to $\mathbf{End}_S$. Let $x$ and $y$ be arbitrary  elements of $S$. Then
\begin{equation}\label{pr2.5_preq2}
g (x \ast y) = g (x) \ast g (y)
\end{equation}
holds by Definition~\ref{def2.5*}. Since $g(x)$ and $g(y)$  belong to $S$ and $f \in \mathbf{End}_S$, Definition~\ref{def2.5*}  and equality~\eqref{pr2.5_preq2} imply
$$
f (g(x \ast y)) = f (g (x) \ast g(y) ) =  f (g (x)) \ast f (g(y)).
$$
Now again using Definition~\ref{def2.5*}, we see that \eqref{pr2.5_preq1} holds. The proof is completed.
\end{proof}

\begin{remark}\label{rem2.14}
Hereinafter we use the symbol $f \circ g$ to denote the mapping
$$
S \overset{g}{\to} S \overset{f}{\to} S.
$$
\end{remark}

Let us recall now the definition of submonoid.

\begin{definition}\label{def2.7}
Let $(S, \ast)$ be a semigroup and $\varnothing \neq T \subseteq S$. Then $T$ is a \emph{subsemigroup} of $S$ if $a, b \in T \Longrightarrow a \ast b \in T$. If $(S, \ast)$ is a monoid with the identity $1_S$, then $T$ is a \emph{submonoid} of $S$ if $T$ is a subsemigroup of $S$ and $1_S \in T$.
\end{definition}

\begin{example}\label{ex2.8}
It follows directly from Definitions~\ref{def2.8} and \ref{def2.3} that $\mathbf{P_U}$ and $\mathbf{P}_{\mathbf{PU}}$ are monoids with respect to the operation of function composition. Moreover using Theorem~\ref{th2.9} and Proposition~\ref{pr2.4} we see that $\mathbf{P_U} \subseteq \mathbf{P}_{\mathbf{PU}}$ and that the identical function $\Id_{\RR^+} : \RR^+ \to \RR^+$,
$$
\Id_{\RR^+} (x) = x \quad \textrm{for each} \quad x \in \RR^+,
$$
is the common identity of the monoids $(\mathbf{P_U}, \circ)$ and $(\mathbf{P_{PU}}, \circ)$,
\begin{equation}\label{ex2.8_eq1}
1_{\mathbf{P_U}} = \Id_{\RR^+} = 1_{\mathbf{P_{PU}}}.
\end{equation}
Thus $(\mathbf{P_U}, \circ)$ is a submonoid of the monoid $(\mathbf{P_{PU}}, \circ)$ by Definition~\ref{def2.7}.
\end{example}

The \emph{kernel} of a monoid homomorphism $F : S \to T $ is the set
\begin{equation}\label{eq_eq1}
Ker(F) := \{ s \in S: F(s) = 1_T\}.
\end{equation}

\begin{proposition}\label{pr2.9}
Let $S=(S, \cdot, 1_S)$ and $T=(T, \ast, 1_T)$ be monoids. Then $Ker (F)$ is a submonoid of the monoid $S$ for every homomorphism $F : S \to T$.
\end{proposition}

\begin{proof}
Let $F : S \to T$ be a homomorphism. By Definition~\ref{def2.5*} we have the equality
$$
F (1_S)  = 1_T.
$$
Thus the membership
\begin{equation}\label{pr2.9_preq2}
1_S \in Ker (F)
\end{equation}
is valid. Hence, by Definition~\ref{def2.7}, $Ker(F)$ is a submonoid of $S$ if
\begin{equation}\label{pr2.9_preq3}
x \cdot y \in Ker (F)
\end{equation}
for all $x,y \in Ker (F)$. Let us consider arbitrary $x,y \in Ker (F)$. Then we have
\begin{equation}\label{pr2.9_preq4}
F(x) = 1_T = F(y)
\end{equation}
by~\eqref{eq_eq1}.

Moreover, since $F$ is a homomorphism, \eqref{pr2.9_preq4} implies
$$
F (x \cdot y) = 1_T \ast 1_T  = 1_T,
$$
i.e., \eqref{pr2.9_preq3} follows.
\end{proof}

\begin{proposition}\label{pr2.10}
Let $S= (S, \ast, 1_S)$ be a monoid and let $(\mathbf{End}_S, \circ)$ be the monoid of all endomorphisms $S \to S$. Then
\begin{equation}\label{pr2.10_eq1}
\mathbf{In}_S :=\{ F\in \mathbf{End}_S : Ker (F) = \{1_S\} \}
\end{equation}
is a submonoid of $(\mathbf{End}_S, \circ)$, where $\{1_S\}$ is the singleton set having only the element $1_S$.
\end{proposition}

\begin{proof}
In Proposition~\ref{pr2.12} it was noted that the equality
$$
1_{\mathbf{End}_S} = \Id_S
$$
holds, where $\Id_S$ is the identical mapping on the set $S$.
In particular, we have
\begin{equation}\label{pr2.10_prs1}
\Id_S (1_S) = 1_S
\end{equation}
and
\begin{equation}\label{pr2.10_prs2}
\Id_S (x) \neq 1_S
\end{equation}
whenever $x \in S$ and $x \neq 1_S$. Now using \eqref{eq_eq1}  we see that \eqref{pr2.10_prs1} and \eqref{pr2.10_prs2} imply
$$
Ker (\Id_S) = \{1_S\}.
$$
Hence the membership
$$
\Id_S \in \mathbf{In}_S
$$
is valid by  \eqref{pr2.10_eq1}. Thus
$$
1_{\mathbf{End}_S} \in \mathbf{In}_S
$$
holds.

Let us consider now arbitrary $F, \Phi \in \mathbf{In}_S$. To complete the proof it suffices to show that
\begin{equation}\label{pr2.10_prs3}
F  \circ \Phi \in \boldsymbol{\In}_S.
\end{equation}
To prove \eqref{pr2.10_prs3} it is suitable to rewrite the equalities
$$
Ker (F) = \{1_S\} \quad \textrm{and} \quad Ker (\Phi) = \{1_S\}
$$
as
\begin{equation}\label{pr2.10_prs4}
\{1_S\} = F^{-1} [\{1_S\}]
\end{equation}
and
\begin{equation}\label{pr2.10_prs5}
\{1_S\} = \Phi^{-1} [\{1_S\}],
\end{equation}
where $F^{-1} [\{1_S\}] (\Phi^{-1} [\{1_S\}])$ is the inverse image of the singleton $\{ 1_S\}$ under mapping $F$ ( mapping $\Phi$).

Now using \eqref{pr2.10_prs4}, \eqref{pr2.10_prs5} and the well-known formula
$$
(F \circ \Phi)^{-1} = \Phi^{-1} \circ F^{-1}
$$
(see, for example \cite{Kur1966}, page 18) we obtain
$$
(F \circ \Phi)^{-1} [\{1_S\}] = \Phi^{-1} (F^{-1} [\{1_S\}]) = \Phi^{-1} [\{1_S\}] = \{1_S\}.
$$
Thus \eqref{pr2.10_prs3} holds. The proof  is completed.
\end{proof}

\section{Main results}

The first our goal is to show that the set $\mathbf{P}_{\mathbf{PU}}$ coincides with the set $\mathbf{End}_{\RR^+}$ of all endomorphisms of the monoid $(\RR^+, \vee)$.

The  following lemma is  simple, and we omit its proof here.

\begin{lemma}\label{lem2.1}
The equivalence
\begin{equation*}\label{lem2.1_eq3}
    (a \leqslant b) \Longleftrightarrow (a \vee b = b)
\end{equation*}
is valid for all $a,b \in \RR^+$.
\end{lemma}

The next theorem is the first main result of the paper.

\begin{theorem}\label{mainth}
The sets $\mathbf{End}_{\RR^+}$ and $\mathbf{P}_{\mathbf{PU}}$ are the same,
\begin{equation}\label{mth_eq1}
    \mathbf{End}_{\RR^+}= \mathbf{P}_{\mathbf{PU}}.
\end{equation}
\end{theorem}

\begin{proof}
Let us prove the inclusion
\begin{equation}\label{mth_preq1}
    \mathbf{End}_{\RR^+} \subseteq \mathbf{P}_{\mathbf{PU}}.
\end{equation}
Suppose that $f \in \mathbf{End}_{\RR^+}$ be arbitrary.
By Proposition~\ref{pr2.4} the membership
\begin{equation}\label{mth_preq2}
    f \in \mathbf{P}_{\mathbf{PU}}
\end{equation}
holds if and only if $f$ is increasing and
\begin{equation}\label{mth_preq3}
    f(0)=0.
\end{equation}

Let us consider arbitrary $a,b \in \RR^+$ satisfying
\begin{equation}\label{mth_preq4}
   a \leqslant b.
\end{equation}

Using Lemma~\ref{lem2.1} we see that \eqref{mth_preq4} holds iff
\begin{equation}\label{mth_preq5}
    a \vee b = b.
\end{equation}
Since $f$ belongs to  $\textbf{End}_{\RR^+}$, equality \eqref{mth_preq5} implies
\begin{equation}\label{mth_preq6}
    f(b) = f( a \vee b) = f(a) \vee f(b).
\end{equation}
Now applying Lemma~\ref{lem2.1}  to \eqref{mth_preq6} we obtain the inequality
$$
f(a) \leqslant f(b).
$$
Thus $f$ is increasing. To prove equality \eqref{mth_preq3} we note that
\begin{equation}\label{mth_preq6_1}
f(1_{\RR^+}) =  1_{\RR^+}
\end{equation}
holds by Definition~\ref{def2.5*}. The last equality and \eqref{ex2.3_eq2} imply \eqref{mth_preq3}. Inclusion \eqref{mth_preq1} follows.

Let us now turn to the proof of the inclusion
\begin{equation}\label{mth_preq7}
   \mathbf{P}_{\mathbf{PU}} \subseteq \textbf{End}_{\RR^+}.
\end{equation}

Let us consider arbitrary $ f \in \mathbf{P}_{\mathbf{PU}} $. To prove \eqref{mth_preq7}, it is enough to show that
\begin{equation}\label{mth_preq8}
   f \in  \textbf{End}_{\RR^+}.
\end{equation}
The last membership holds iff
\begin{equation}\label{mth_preq9}
  f(a \vee b) = f(a) \vee f(b),
\end{equation}
for all $a, b \in \RR^+$, and $f$ satisfies~\eqref{mth_preq6_1}.
By Proposition~\ref{pr2.4},
$$
f \in \mathbf{P}_{\mathbf{PU}}
$$
implies that $f$ is increasing and satisfies
$$
f (0) = 0.
$$
The last equality  and \eqref{ex2.3_eq2} imply \eqref{mth_preq6_1}.
Thus to complete the proof of \eqref{mth_preq7} it suffices to show that \eqref{mth_preq9} holds for all $a,b \in \RR$. Let us do it.

Without loss of generality, we can assume that \eqref{mth_preq4} holds.
By Lemma~\ref{lem2.1}, inequality \eqref{mth_preq4} holds iff
$$
a \vee b = b.
$$
Hence we have
\begin{equation}\label{mth_preq12}
  f( a \vee b) = f (b).
\end{equation}
Since $f$ is increasing, we also obtain the inequality
$$
f(a) \leqslant f(b),
$$
that implies
\begin{equation}\label{mth_preq13}
   f(b) \vee f(a) = f(b)
\end{equation}
by Lemma~\ref{lem2.1}. Now \eqref{mth_preq9} follows from \eqref{mth_preq12} and \eqref{mth_preq13}. Membership relation \eqref{mth_preq8} is proven.
Inclusion~\eqref{mth_preq7} follows.

Since \eqref{mth_preq7} and \eqref{mth_preq1} imply \eqref{mth_eq1}, the proof is completed.
\end{proof}

Proposition~\ref{pr2.12} implies that the set $\mathbf{End}_{\RR^+}$
 together with the operation of function composition is a monoid. The same operation is the multiplication on $\mathbf{P}_{\mathbf{PU}} $. Thus we obtain the following.

\begin{corollary}\label{cor3.3}
The monoids $(\mathbf{End}_{\RR^+}, \circ)$ and $(\mathbf{P}_{\mathbf{PU}}, \circ)$ are the same.
\end{corollary}

The next result is an analog of Corollary~\ref{cor3.3}  for the monoid $(\mathbf{P_U}, \circ)$. Let us denote by $(\mathbf{In}_{\RR^+}, \circ)$ the submonoid of $(\mathbf{End}_{\RR^+}, \circ)$ consisting of all endomorphisms $F : \RR^+ \to \RR^+$ satisfying the equality
$$
Ker (F) = \{1_{\RR^+}\}.
$$

\begin{theorem}\label{th3.4}
The monoids $(\mathbf{In}_{\RR^+}, \circ)$ and $(\mathbf{P_U}, \circ)$  are the same.
\end{theorem}

\begin{proof}
It was noted in Example~\ref{ex2.11} that $(\mathbf{P}_{\mathbf{U}}, \circ)$ is a submonoid  of the monoid $(\mathbf{P}_{\mathbf{PU}}, \circ)$. Similarly, by Proposition~\ref{pr2.10}, $(\mathbf{In}_{\RR^+},\circ)$ is a submonoid of the monoid $(\mathbf{End}_{\RR^+}, \circ)$. Moreover, the equality
\begin{equation}\label{th3.4_preq1_1}
(\mathbf{End}_{\RR^+}, \circ) = (\mathbf{P_{PU}}, \circ)
\end{equation}
holds by Corollary~\ref{cor3.3}. Thus it suffices to show that the sets $\mathbf{P_U}$ and $\mathbf{In}_{\RR^+}$  are the same,
\begin{equation}\label{th3.4_preq1}
\mathbf{P_{U}} = \mathbf{In}_{\RR^+}.
\end{equation}

Let us consider an arbitrary function $g: \RR^+ \to \RR^+$. Theorem~\ref{th2.9} and Proposition~\ref{pr2.4} imply that $g$ is ultrametric-preserving iff $g$ is pseudoultrametric-preserving and amenable.

Hence the equality
\begin{equation}\label{th3.4_preq2}
\mathbf{P_U}= \{ f \in \mathbf{P_{PU}} : f^{-1}(0)= \{0\} \}
\end{equation}
holds. Equality~\eqref{ex2.3_eq2} implies that
\begin{equation}\label{th3.4_preq4}
 \{0\}  = \{ 1_{\RR^+} \}.
\end{equation}
Now using \eqref{th3.4_preq1_1} and \eqref{th3.4_preq4} we can rewrite \eqref{th3.4_preq2} as
\begin{equation}\label{th3.4_preq5}
\mathbf{P_U}= \{ f \in \mathbf{End}_{\RR^+} : f^{-1}(1_{\RR^+})= \{1_{\RR^+}\} \}.
\end{equation}
It follows from~\eqref{eq_eq1}  with $S=(\RR^+, \vee) = T$ that the equality
$$
f^{-1} (1_{\RR^+}) = \{ 1_{\RR^+}\}
$$
holds if and only if
$$
Ker (f) = \{ 1_{\RR^+}\}.
$$
Hence~\eqref{th3.4_preq5} is equivalent to the equality
\begin{equation}\label{th3.4_preq6}
\mathbf{P_U}= \{ f \in \mathbf{End}_{\RR^+} : Ker (f) = \{ 1_{\RR^+}\} \}.
\end{equation}
Equality~\eqref{pr2.10_eq1} with $S = (\RR^+, \vee)$ gives us
\begin{equation}\label{th3.4_preq7}
\mathbf{In}_{\RR^+}= \{ 1 \in \mathbf{End}_{\RR^+} : Ker (f) = \{ 1_{\RR^+}\} \}.
\end{equation}
Now \eqref{th3.4_preq1} follows from \eqref{th3.4_preq6} and \eqref{th3.4_preq7}. The proof is completed.
\end{proof}

The set $\mathbf{P_U}$ of ultrametric-preserving functions was also studied in \cite{Dov24TA, BDa2024, BDS2021} and \cite{VD2021MS}. In particular, it was proved in Theorem~31 of \cite{BDa2024} that for every submonoid $\mathbf{A}$ of the monoid $(\mathbf{P_U}, \circ)$ there is a subclass $\mathbf{X}$ of the class $\mathbf{U}$ of all ultrametric spaces such that
$\mathbf{P_X} = \mathbf{A},$ where $\mathbf{P_X}$ is the set introduced in Definition~\ref{def1.8}.  The next our goal is to generalize this result  on the submonoids of the monoid $(\mathbf{P_{PU}}, \circ)$.

\begin{lemma}\label{lem3.5}
Let $\mathbf{X} \subseteq \mathbf{PU}$ be nonempty. Then $\mathbf{P_X}$ is a monoid with respect to the operation of function composition and, in addition, the identical mapping $\Id_{\mathbb{R}^+} : \RR^+ \to \RR^+$ is the identify of the monoid $(\mathbf{P_X}, \circ)$,
\begin{equation}\label{lem3.5_eq1}
1_{\mathbf{P_X}} = \Id_{\RR^+}.
\end{equation}
\end{lemma}

\begin{proof}
It follows from Definition~\ref{def1.8}. Indeed, let us consider arbitrary $f,g \in \mathbf{P_X}$ and $(Y, \rho ) \in \mathbf{X}$. Then we have
$$
(Y, g \circ \rho) \in \mathbf{X}
$$
by Definition~\ref{def1.8}. Applying this definition again we obtain
$$
(Y, f \circ (g \circ \rho)) \in \mathbf{X}.
$$
Since $(f \circ g) \circ \rho = f \circ(g \circ \rho)$ holds, Definition~\ref{def1.8} also implies $(f \circ g) \in \mathbf{P_X}$. It is also easy to see that
$$
(Y, \Id_{\RR^+}  \circ \rho) = (Y,  \rho) \quad \textrm{and} \quad
\Id_{\RR^+} \circ f = f \circ \Id_{\RR^+} = f.
$$
Thus $(\mathbf{P_X}, \circ)$ is a monoid and equality \eqref{lem3.5_eq1} holds.
The proof is completed.
\end{proof}

\begin{theorem}\label{th3.5}
    Let $\mathbf{A}$ be a nonempty subset  of $\mathbf{P_{PU}}$. Then the following statements  are equivalent.
    \begin{itemize}
        \item[(i)] $\mathbf{A}$ is a submonoid of $(\mathbf{P_{PU}},  \circ)$.

        \item[(ii)] $\mathbf{A}$ is a submonoid of $(\mathbf{End}_{\RR^+},  \circ)$.

        \item[(iii)] There is a nonempty $\mathbf{X} \subseteq \mathbf{PU}$ that satisfies
        \begin{equation}\label{th3.5_eq1}
        \mathbf{P_X} = \mathbf{A}.
        \end{equation}
    \end{itemize}
\end{theorem}

\begin{proof}
The validity of the equivalence $(i) \Longleftrightarrow (ii)$ follows  from Corollary~\ref{cor3.3}.

Let us prove the validity of the implication $(i) \Longrightarrow (iii)$.

Suppose that $\mathbf{A}$ is a submonoid  of the monoid $(\mathbf{P_{PU}}, \circ)$. Let $d^{+}$ be the ultrametric on $\RR^+$ defined in Example~\ref{ex1.2}. Then using formula \eqref{ex1.2_eq1} it is easy to see that  $d^+$ is  a surjective mapping on $\RR^+ \times \RR^+$,
\begin{equation}\label{th3.5_preq2}
\RR^+ = \{ d(p,q) : p,q \in \RR^+\}.
\end{equation}
Since $\mathbf{A}$ is a submonoid of $(\mathbf{P_{PU}}, \circ)$, the inclusion
$$
\mathbf{A} \subseteq \mathbf{P_{PU}}
$$
holds. Consequently, for every $f \in \mathbf{A}$, $(\RR^+, f \circ d^+)$ is a pseudoultrametric space by Definition~\ref{def2.3}.

Write
\begin{equation}\label{th3.5_preq3}
\mathbf{X} : = \{ (\RR^+, g \circ d^+) : g \in \mathbf{A}\}.
\end{equation}
Since $(\RR^+, d^+)$ belongs $\mathbf{U}$ and $\mathbf{A} \subseteq \mathbf{P_{PU}}$, we have
$$
\mathbf{X} \subseteq \mathbf{PU}.
$$
We claim  that $\mathbf{X}$  satisfies  equation \eqref{th3.5_eq1}.

To prove the claim we must show that  for each $f: \RR^+ \to \RR^+$ the equivalence
\begin{equation}\label{th3.5_preq4}
(f \in \mathbf{A}) \Longleftrightarrow \left( (\RR^+, f \circ (g \circ d^+) ) \in \mathbf{X} \right)
\end{equation}
is valid for every $g \in \mathbf{A}$. (See Definition~\ref{def1.8}).

Let us consider an arbitrary $f \in \mathbf{A}$. Then $f \circ g \in \mathbf{A}$ holds for every $g \in \mathbf{A}$ because $\mathbf{A}$ is a submonoid of $(\mathbf{P_{PU}}, \circ)$. Thus, for every $g \in \mathbf{A}$,
\begin{equation}\label{th3.5_preq5}
 \left( \RR^+, (f \circ g) \circ d^+ \right) \in \mathbf{X}
\end{equation}
holds by \eqref{th3.5_preq3}. Since the function composition is associative, we may rewrite \eqref{th3.5_preq5} as
$$
\left( \RR^+, f \circ (g \circ d^+) \right) \in \mathbf{X} .
$$
Consequently, the implication
$$
(f \in \mathbf{A}) \Longrightarrow \left( (\RR^+, f \circ (g \circ d^+) ) \in \mathbf{X} \right)
$$
is valid for every $g \in \mathbf{A}$.

Let us consider now an arbitrary $f : \RR^+ \to \RR^+$ and suppose that the membership
\begin{equation}\label{th3.5_preq6}
 \left( \RR^+, f \circ (g \circ d^+) \right) \in \mathbf{X}
\end{equation}
is valid for every $g \in \mathbf{A}$. We must show that
\begin{equation}\label{th3.5_preq7}
f \in \mathbf{A}.
\end{equation}
To prove the validity of \eqref{th3.5_preq7} we note the identical mapping $\Id_{\RR^+}$ belongs to $\mathbf{A}$ by \eqref{ex2.8_eq1}. Consequently \eqref{th3.5_preq6} with $g = \Id_{\RR^+}$ implies
\begin{equation}\label{th3.5_preq8}
 \left( \RR^+, (f \circ  d^+) \right) \in \mathbf{X}.
\end{equation}
It follows from \eqref{th3.5_preq8} and \eqref{th3.5_preq3} that there is $g \in \mathbf{A}$ such that
$$
g \circ d^+ = f \circ d^+.
$$
The last equality and \eqref{th3.5_preq2} imply the equality $g =f$. Consequantly \eqref{th3.5_preq7} is valid. Thus equivalence \eqref{th3.5_preq4} is valid for every $g \in \mathbf{A}$ as required.

Let us prove the validity of the implication $(iii) \Longrightarrow (i)$.

Suppose that there is a nonempty $\mathbf{X} \subseteq \mathbf{PU}$ satisfying~\eqref{th3.5_eq1}. We must show that $\mathbf{A}$ is a submonoid of $(\mathbf{P_{PU}}, \circ)$.

Lema~\ref{lem3.5} and equality~\eqref{s3_eq2} imply that $\mathbf{A}$ is a monoid. Using Definition~\ref{def2.7}, equality~\eqref{lem3.5_eq1} and the inclusion $\mathbf{A} \subseteq \mathbf{P_{PU}}$ we obtain that $\mathbf{A}$ is a submonoid of $(\mathbf{P_{PU}}, \circ)$.

The proof is completed.
\end{proof}

Analysing the above proof, we get the following.

\begin{proposition}\label{prop3.6}
Let $\mathbf{A}$ be a submonoid of the monoid $\mathbf{P_{PU}}$ and let
\begin{equation*}\label{pr3.6_eq1}
\mathbf{X} := \{(\RR^+, g \circ d^+)  : g \in \mathbf{A}\}.
\end{equation*}
Then the equality
\begin{equation*}\label{pr3.6_eq2}
    \mathbf{P_X} = \mathbf{A}
\end{equation*}
holds.
\end{proposition}

In the remainder of this section we will formulate some analogies to  Proposition~\ref{prop3.6}.

\begin{definition}\label{def3.7}
A nonempty subset $R$ of a semigroup $(S, \ast)$ is called a \emph{right ideal} of $(S, \ast )$ if
$$
r \ast s \in R
$$
holds whenever $r \in R$ and $s \in S$.
\end{definition}

\begin{definition}\label{def3.6}
Let $\mathbb{N} = \{1,2,3, \ldots \}$, $(S, \ast)$ be a semigroup and let $A$ be a nonempty subset of $S$. Then
\begin{equation}\label{def3.6_eq1}
[A]_S := \{ x_1 \ast x_2 \ast \ldots \ast x_n : \ x_1, x_2, \ldots, x_n \in A, \  n \in \mathbb{N} \}
\end{equation}
is a subsemigroup of $(S, \ast)$, called  the \emph{subsemigroup generated by} $A$.
\end{definition}

Definition~\ref{def3.6}, in particular, shows that $[A]_S$ is the smallest subsemigroup of $(S, \ast)$ containing the set $A$.

In the following lemma we describe the largest right ideal $\overline{R}$ of $[A]_S$ containing in $A$.

\begin{lemma}\label{lem3.8}
Let $(S, \ast)$ be a semigroup and let $A \subseteq S$ be nonempty. Denote by $R_A$ the set of all right ideals $R$ of  $[A]_S$ satisfying
\begin{equation}\label{lem3.8_eq1}
R \subseteq A.
\end{equation}
 Then the set
\begin{equation}\label{lem3.8_eq2}
\overline{R} := \bigcup\limits_{R\in R_A} R,
\end{equation}
is either empty
\begin{equation}\label{lem3.8_eq3}
\overline{R} = \emptyset
\end{equation}
or it is a right ideal of the semigroup $[A]_S$.
\end{lemma}

\begin{proof}
If $R_A = \emptyset$, then \eqref{lem3.8_eq3} evidently holds. Suppose that $R_A \neq \emptyset$, then the set $\overline{R}$ is also nonempty,
$$
\overline{R} \neq \emptyset.
$$
Let us consider arbitrary $r \in \overline{R}$ and $s\in S$. We must show that
\begin{equation}\label{lem3.8_preq1}
r \ast s  \in \overline{R}.
\end{equation}
In order to do this, it is enough to note that $r \in \overline{R}$ implies that $r \in R$ for some $R \in R_A$ by \eqref{lem3.8_eq2}. Hence
$$
r \ast s \in R \quad \textrm{and} \quad R \subseteq \overline{R},
$$
hold, that implies \eqref{lem3.8_preq1}. Thus $\overline{R}$ is a right ideal of the semigroup $[A]_S$.
\end{proof}

If a semigroup $(S, \ast)$ has no identity element then it is easy to adjoin an element $1 \neq S$ to form a monoid $S^1$. (See, for example, \cite{Howie2003} page 2). Below we use a special modification of this procedure.

Let $M= (M, \ast, 1_M)$ be a monoid and let $S \subseteq M$ be either empty or be a subsemigroup of the monoid $M$. We can expand $S$ to a smallest submonoid $S^{1_M}$ of $M$ as
\begin{equation}\label{s3_eq2}
    S^{1_M} := \left\{
  \begin{array}{ll}
    S, \quad & \hbox{if} \quad 1_M \in S, \\
    S \cup \{1_M\}, \quad & \hbox{otherwise.}
  \end{array}
\right.
\end{equation}
In particular, the equality
$$
S^{1_M} = \{1_M\}
$$
holds if and only if we have either $S=\{1_M\}$ or $S=\emptyset$.
It should be also noted here that, in general, $S^{1_M} \neq S^1$ even if $S$ is a monoid.

\begin{example}\label{ex3.10}
   Let $A=\{a\}$  be an one-point subset of $\mathbb{R}^+$. Then $A=(A, \vee)$ is a monoid and this monoid is a subsemigroup of the monoid $(\mathbb{R}^+, \vee)$ but Definition~\ref{def2.7} implies that $A$ is a submonoid of $(\RR^+,\vee)$ if and only if $a=0$. Indeed, using \eqref{s3_eq2} it is easy to see that
   $$
 A^{1_{\RR^+}} = \{0,a\} \neq \{a\}= A^1
   $$
   whenever $a>0$.
\end{example}

The next proposition is an equivalent to Proposition~\ref{prop3.6}.

\begin{proposition}\label{prop3.11}
    Let $\mathbf{A}$ be a subset of the monoid $\mathbf{P_{PU}}= (\mathbf{P_{PU}}, \circ, 1_{\mathbf{P_{PU}}})$,
    \begin{equation*}\label{pr3.11_eq1}
        \mathbf{X} := \{ (\RR^+, g \circ d^+) : g \in\mathbf{A} \}
    \end{equation*}
    and let $R_{\mathbf{A}}$ be the set of all right ideals  of $[\mathbf{A}]_{\mathbf{P_{PU}}}$. Write
    \begin{equation*}\label{pr3.11_eq2}
        \overline{R} := \bigcup\limits_{R \in R_{\mathbf{A}}} {R}.
    \end{equation*}
    If
    \begin{equation}\label{pr3.11_eq3}
        1_{\mathbf{P_{PU}}} \in \mathbf{A} \cap \overline{R},
    \end{equation}
    then the equality
    \begin{equation}\label{pr3.11_eq4}
        \overline{R}^{1_{\mathbf{P_{PU}}}} = \mathbf{P}_{\mathbf{X}}
    \end{equation}
    holds.
\end{proposition}

\begin{proof}
    Let \eqref{pr3.11_eq3} hold. We must prove equality~\eqref{pr3.11_eq4}.

   We claim that
    \begin{equation}\label{pr3.11_preq2}
        \overline{R} = \overline{R}^{1_{\mathbf{P_{PU}}}}.
    \end{equation}
    Indeed, \eqref{pr3.11_eq3} implies that $\overline{R}$ is a nonempty set. Consequently $\overline{R}$ is a right ideal of $[\mathbf{A}]_{\mathbf{P_{PU}}}$ by Lemma~\ref{lem3.8}. Since every right ideal of a semigroup is a subsemigroup of this semigroup, $\overline{R}$ is a subsemigroup of $[\mathbf{A}]_{\mathbf{P_{PU}}}$. The semigroup $[\mathbf{A}]_{\mathbf{P_{PU}}}$ is a subsemigroup of the monoid $\mathbf{P_{PU}}$. Hence $\overline{R}$ also is a subsemigroup  of $\mathbf{P_{PU}}$.  Now \eqref{pr3.11_preq2} follows from \eqref{s3_eq2} with
    $$
      S = \overline{R} \quad \textrm{and} \quad M = \mathbf{P_{PU}}.
    $$

    Using \eqref{lem3.8_eq1} with $A=\mathbf{A}$ we see that  $\overline{R}$ is a subset of  the set $\mathbf{A}$,
    $$
      \overline{R} \subseteq \mathbf{A}.
    $$
    In addition, condition \eqref{pr3.11_eq3} and Lemma~\ref{lem3.8} imply
    $$
      \mathbf{A} \subseteq \overline{R}.
    $$
    Thus we have the equality
     \begin{equation}\label{pr3.11_preq3}
       \mathbf{A} = \overline{R}.
    \end{equation}

    Let us prove equality \eqref{pr3.11_eq4}. It is shown above that $\overline{R}$ is a submonoid of $(\mathbf{P_{PU}}, \circ, 1_{\mathbf{P_{PU}}})$.  Hence $\mathbf{A}$ also is a submonoid of $(\mathbf{P_{PU}}, \circ, 1_{\mathbf{P_{PU}}})$ by \eqref{pr3.11_preq3}.  Now using Proposition~\ref{prop3.6} we obtain
    $$
      \mathbf{P_X}  = \mathbf{A},
    $$
    that implies \eqref{pr3.11_eq4} by \eqref{pr3.11_preq2} and \eqref{pr3.11_preq3}.
\end{proof}

We conclude the paper by following conjecture.

\begin{conjecture}\label{con3.14} {\rm (Prove or disprove)}
 Let $\mathbf{A}$ be a subset of the monoid $\mathbf{P_{PU}}= (\mathbf{P_{PU}}, \circ, 1_{\mathbf{P_{PU}}})$, $
        \mathbf{X} := \{ (\RR^+, g \circ d^+) : g \in\mathbf{A} \} $,
    and let $R_{\mathbf{A}}$ be the set of all right ideals  of $[\mathbf{A}]_{\mathbf{P_{PU}}}$. Write
    \begin{equation*}\label{th3.11_eq2}
        \overline{R} := \bigcup\limits_{R \in R_{\mathbf{A}}} {R}.
    \end{equation*}
    If $1_{\mathbf{P_{PU}}}$ belongs to $\mathbf{A}$, then the equality
    \begin{equation*}\label{th3.11_eq4}
        \overline{R}^{1_{\mathbf{P_{PU}}}} = \mathbf{P}_{\mathbf{X}}
    \end{equation*}
    holds.
\end{conjecture}

\section*{Conflict of interest statement}

The research was conducted in the absence of any commercial or financial relationships that could be construed as a potential conflict of interest.

\section*{Funding}

The author was supported by grant 359772 of the Academy of Finland.

\end{document}